\documentclass[11pt,a4paper,oneside,english]{amsart}
\usepackage{amstext}
\usepackage{amsthm}
\usepackage{libertine-type1}
\usepackage[scaled=0.95]{biolinum-type1}
\usepackage[scaled=0.95]{libertineMono-type1}
\usepackage[libertine]{newtxmath}
\usepackage[T1]{fontenc}
\usepackage[latin9]{inputenc}
\usepackage{bm}
\usepackage{graphicx}
\usepackage{setspace}
\makeatletter

\pdfpageheight\paperheight
\pdfpagewidth\paperwidth

\numberwithin{equation}{section}
\numberwithin{figure}{section}
\theoremstyle{plain}
\newtheorem{thm}{\protect\theoremname}
  \theoremstyle{plain}
  \newtheorem{lem}[thm]{\protect\lemmaname}

\makeatother

\usepackage{babel}
  \providecommand{\lemmaname}{Lemma}
\providecommand{\theoremname}{Theorem}

\begin{document}

\title[An active-set mixed finite element solver for a lubrication problem]{An active-set mixed finite element solver for a transient hydrodynamic
lubrication problem in the presence of cavitation. }

\author{moulay hicham tber}

\date{\today}

\address{Cadi Ayyad University, F.S.T.G. B.P 549, Department of Mathematics,
Av. Abdelkarim Elkhattabi, Marrakech, Morocco.}



\email{hicham.tber@uca.ac.ma}
\begin{abstract}
In this paper we study a moving free boundary problem related to the
the cavitation modeling in lubricated devices. More precisely, a characteristics
method combined with a weak formulation in a mixed form is introduced
for the Elrod-Adams model. The formulation is suitable for the use
of mixed finite element methods in the numerical approximation. It
is proved that the time-discrete problem and its finite element discretization
has a unique solution. Further an efficient primal-dual active-set
strategy is proposed to solve the resulting algebraic system. The
performance of the overall algorithm is illustrated by numerical examples.
\end{abstract}

\maketitle

\section{Introduction}
For a long time, Reynolds equation has been used to describe the behavior
of a viscous flow inside the lubricated devices. Nevertheless, Reynolds
modeling approach does not take into account the rupture of the continuous
lubricant film due to the formation of air bubbles. This phenomenon,
called cavitation, is not only important because its onset and extent
determine the performance of the lubricated device but also because
vapor cavitation collapse (implosion) can cause severe surface material
damage \cite{San-Andreas}. Thus, various models have been used in
order to make the Reynolds equation valid in the cavitation area.
The Elrod-Adams modeling approach \cite{ElrodAdams}, here adopted,
is one of the most realistic approaches taking account of the cavitation
phenomena. In this model, the cavitation region is considered as a
fluid-air mixture and an additional unknown representing the saturation
of the fluid in the mixture is introduced. The study of this model
has given place to many works covering some fundamental and applied
works (see \cite{Bay-Vaz,Gropper} and the references within).

Our motivation in this paper is to design an efficient numerical solution
algorithm based on an appropriate time-space discretization for the
Elrod-Adams cavitation model. By considering flow along the characteristics
we derive a mixed variational formulation of the problem. We rewrite
the semi-discrete governing equations in the form of system of first
order equations in which the flux is treated as an independent variable.
We mention here that a such approach has been already used for advection-dominated
transport problems \cite{Arbogast}. Here it is shown that the semi-discrete
problem has a unique solution by using Shauder's fixed point theorem
and a regularization technique. Moreover, the weak solution is approximated
by a mixed finite element (MFE) method. A primal-dual active set strategy
which is equivalent to a semi-smooth Newton algorithm is used to solve
the complementarity-saddle point problem arising from the characteristics-MFE
discretization. \\

The rest of this paper is organized as follows: In section 2 the mathematical
formulation of the problem is presented. The set of equations defines
a nonlinear advection diffusion free boundary problem where the free
boundary separates the lubricated and cavitated regions. In section
3, we approximate the hyperbolic part of the equation along the characteristics.
In section 4, we derive a weak formulation of the semi-discrete problem
in a mixed form and we prove the existence of one weak solution. In
the same section we approximate the weak solution by a mixed finite
element method. In section 5, we design a solution algorithm for the
resulting nonlinear algebraic equation. Numerical experiments are
reported in section 6.

\section{the mathematical model}

Studying the cavitation phenomenon in lubricated devices gives place
to a mathematical formulation in a domain $\Omega\subset\mathbb{R}^{n}$
with $n=1,\,2$ of regular boundary $\Gamma=\Gamma_{i}\cup\Gamma_{e}$
with $\Gamma_{i}\cap\Gamma_{e}=\emptyset$. According to Elrod-Adams
model, the unknowns of the problem are the pressure $p(x,t)$ and
the relative content $\theta(x,t)$ of oil film with $x\in\Omega$
and $t_{0}\leq t\leq T$. When the lubrication takes place by an incompressible
fluid of viscosity $\mu$, the pressure satisfies the Reynolds equation:
\begin{equation}
\dfrac{\partial h}{\partial t}-\nabla\left(\dfrac{h^{3}}{12\mu}\nabla p-\dfrac{h}{2}U\bm{e_{1}}\right)=0,\quad\theta=1,\quad\textrm{ in }\Omega^{+}=\{x\in\Omega\;|\;p(x)>p_{c}\}\label{ReynoldsEquation}
\end{equation}
and $\theta$ satisfies the conservation law 
\begin{equation}
\dfrac{\partial (h\theta)}{\partial t} + \nabla\left(\dfrac{h}{2}\theta U\bm{e_{1}}\right)=0,\qquad0\leq\theta\leq1,\qquad\textrm{in}\quad\Omega^{0}=\{x\in\Omega\;|\;p(x)=p_{c}\},\label{ConservationLaw}
\end{equation}
where $p_{c}\in\mathbb{R}$ is the cavitation pressure, $\bm{e_{1}}$
is the unit vector in the $x_{1}$-direction and $U(x,t)$ is the
relative sliding speed of the contact surfaces with $U\bm{e_{1}}$
being a divergence-free vector field. The function $h=h(x,\,t)$ that
represents the film thickness satisfies 
\begin{equation}
\exists h_{0},\,h_{1}\in\mathbb{R}_{+}^{*},\qquad\forall(x,\,t)\in\Omega\times\left]t_{0},\,T\right[,\qquad h_{0}\leq h(x,\,t)\leq h_{1}.\label{thikness}
\end{equation}
On the free boundary $\Sigma=\partial\Omega^{+}\cap\Omega,$ we have the conservation condition of the flux 
\[
(h^0 \theta^0 - h^+) \dfrac{U}{2}\bm{e_{1}}\cdot \mathbf{n} + \dfrac{(h^+)^{3}}{12\mu}\left(\dfrac{\partial p}{\partial\mathbf{n}}\right)^+ = (h^0 \theta^0 - h^+)\mathbf{V_\Sigma}\cdot\mathbf{n},
\]
where $\mathbf{n}$ is the unit vector normal on $\Sigma$ oriented outwards from $\Omega^+.$ $\mathbf{V_\Sigma}$ is the velocity of $\Sigma.$ The superscripts $0$ and $+$ refers to the limit values of the corresponding quantities as $\Sigma$ is approached from the cavitated and full film sides respectively.

The Reynolds equation $\eqref{ReynoldsEquation}$ and the conservation
law $\eqref{ConservationLaw}$ lead to the following equation valid
in both cavitation and lubricated region: 
\[
\dfrac{\partial\theta h}{\partial t}+\nabla\left(\dfrac{h}{2}\theta U\bm{e_{1}}\right)-\nabla(\dfrac{h^{3}}{12\mu}\nabla p)=0.
\]
The pressure $p$ and the concentration $\theta$ are related by 
\[
p\geq0\qquad0\leq\theta\leq1\qquad p(1-\theta)=0.
\]
Here and in the following, it is assumed without a loss of generality
that the cavtitaion pressure $p_{c}=0.$ Thus the strong formulation
of the problem is given by the following set of equations: 

\begin{equation}
\left\{ \begin{array}{c}
\dfrac{\partial(h\theta)}{\partial t}+\dfrac{U}{2}\bm{e_{1}}\cdot\nabla(h\theta)-\nabla\left(\dfrac{h^{3}}{12\mu}\nabla p\right)=0,\\
p\geq0,\quad p(1-\theta)=0,\quad0\leq\theta\leq1,\\
p|_{\Gamma_{i}}=p_{i},\qquad p|_{\Gamma_{e}}=p_{e},
\end{array}\right.\label{StrongForm}
\end{equation}
where $p_{i}=p_{i}(t)$ and $p_{e}=p_{e}(t)$ are given supply pressures.
\\
We supplement the above system by an appropriate initial condition
$\theta(t=t_{0})=\theta_{0}$ with $0\leq\theta_{0}(x)\leq1$ for
$x\in\Omega.$ 

\section{time discretization}

Due to the hyperbolic character of the governing equation in cavitated
areas, the numerical solutions of (\ref{StrongForm}) may exhibit
undesired oscillations. Following \cite{Bay-et-al}, one possible
solution to deal with this issue consists in writing $\dfrac{\partial(h\theta)}{\partial t}+\dfrac{U}{2}\bm{e_{1}}\cdot\nabla(h\theta)$
as $\dfrac{D(h\theta)}{Dt}$ the material derivative of $h\theta$
in the direction of $\dfrac{U}{2}\bm{e_{1}}.$ Then we can rewrite
the first equation in (\ref{StrongForm}) as 
\begin{equation}
\dfrac{D(h\theta)}{Dt}-\nabla\left(\dfrac{h^{3}}{12\mu}\nabla p\right)=0.\label{EqTotalDiv}
\end{equation}
The corresponding characteristic curves are defined by 
\[
\left\{ \begin{array}{c}
\dfrac{dX(x,\,t;\,s)}{ds}=\dfrac{U(x,\,s)}{2}\bm{e_{1}},\\
X(x,\,t;\,t)=x,
\end{array}\right.
\]
with $X(x,t;s)$ being the position of a particle at time $s,$ which
was or will be at $x$ at time $t$. \\
Now for a given uniform time step size $\tau>0,$ we can get an approximate
value of $X$ at $t^{old}=t-\tau$ by
\[
X(x,\,t;\,t^{old})=x-\tau\,\dfrac{U(x,\,t)}{2}\bm{e_{1}}.
\]
Using a fully-implicit scheme, we get the semi-discrete form of (\ref{StrongForm})
\begin{equation}
\left\{ \begin{array}{c}
h\theta-\nabla\left(\dfrac{\tau h^{3}}{12\mu}\nabla p\right)=h^{old}\theta^{old}\\
p\geq0,\quad p(1-\theta)=0,\quad0\leq\theta\leq1,\\
p|_{\Gamma_{i}}=p_{i},\qquad p|_{\Gamma_{e}}=p_{e},\\
\theta(t=0)=\theta_{0},
\end{array}\right.\label{time-discrete-eq}
\end{equation}
where $\theta^{old}(x,\,t)=\theta\left(X(x,\,t;\,t^{old}),\,t^{old}\right)$
and $h^{old}(x,\,t)=h\left(X(x,\,t;\,t^{old}),\,t^{old}\right).$
This time approximation will be combined with a spatial discretization
by a mixed finite element method. 

\section{mixed formulation }

In this section, we introduce a weak formulation for the semi-discrete
problem (\ref{time-discrete-eq}) in a mixed form. First we introduce
the Hilbert space
\[
H(div)=\{{\bf v}\in L^{2}(\Omega)^{n},\;\nabla\cdot\mathbf{v}\in L^{2}(\Omega)\}
\]
and the associated norm 
\[
\Vert\mathbf{v}\Vert_{H(div)}=\sqrt{\Vert\mathbf{v}\Vert_{0}^{2}+\Vert\nabla\cdot\mathbf{v}\Vert_{0}^{2}},
\]
where $\Vert\cdot\Vert_{0}$ denotes the standard $L^{2}$ norm for
either vector-valued or real-valued functions as appropriate. \\
Next we define the velocity $\mathbf{u}$ by the following relation:
\begin{equation}
\dfrac{12\mu}{\tau h^{3}}\mathbf{u}+\nabla p=0.\label{DefinitionVelocity}
\end{equation}
It follows from (\ref{time-discrete-eq}) that 
\begin{equation}
h\theta+\nabla\cdot\mathbf{u}=h^{old}\theta^{old}.\label{DivEq}
\end{equation}
By testing $\eqref{DefinitionVelocity}$ against any vector valued
$\mathbf{v}\in H(div),$ we obtain 
\begin{equation}
(\dfrac{12\mu}{\tau h^{3}}\mathbf{u},\;\mathbf{v})+(\nabla p,\;\mathbf{v})=0.
\end{equation}
Using Green's formula, one arrives at 
\begin{equation}
(\dfrac{12\mu}{\tau h^{3}}\mathbf{u},\;\mathbf{v})-(\nabla\cdot\mathbf{v},\,p)=-<p_{\Gamma},\,\mathbf{v}\cdot n>_{\Gamma_{D}},\label{EqU}
\end{equation}
where $(\cdot,\,\cdot)$ and $<\cdot,\,\cdot>_{\Gamma}$ indicates
the inner product on $L^{2}(\Omega)$ and $L^{2}(\Gamma)$ respectively
and $p_{\Gamma}$ denotes the trace of the pressure $p$ on the boundary
$\Gamma:$
\[
p_{\Gamma}=\left\{ \begin{array}{ll}
p_{i} & \quad\textrm{ on }\quad\Gamma_{i},\\
p_{e} & \quad\textrm{ on }\quad\Gamma_{e}.
\end{array}\right.
\]
Finally, by testing $\eqref{DivEq}$ against any $q\in L^{2}(\Omega)$
we obtain 
\begin{equation}
(h\theta,\,q)+(\nabla\cdot\mathbf{u},\,q)=(h^{old}\theta^{old},\,q).\label{EqP}
\end{equation}
Hence, the weak formulation of the problem consists in finding $p\in L^{2}(\Omega),$
$\theta\in L^{\infty}(\Omega)$ and $\mathbf{u}\in H(div)$ such that 

\begin{equation}
\left\{ \begin{array}{l}
(\dfrac{12\mu}{\tau h^{3}}\mathbf{u},\,\mathbf{v})-(\nabla\cdot\mathbf{v},\,p)=-<p_{\Gamma},\,\mathbf{v}\cdot n>_{\Gamma}\qquad\forall\mathbf{v}\in H(div),\\
(\nabla\cdot\mathbf{u},\,q)+(h\theta,\,q)=\left(h^{old}\theta^{old},\,q\right)\qquad\forall q\in L^{2}(\Omega),\\
p\geq0,\quad p(1-\theta)=0,\quad0\leq\theta\leq1,\qquad\textrm{ a.e. in }\Omega.
\end{array}\right.\label{WeakMixedProblem}
\end{equation}

To show that the weak formulation has a solution, let $H_{\varepsilon}$
be a regularization of the Heaviside operator $H$ given by 
\[
H_{\varepsilon}(x)=\left\{ \begin{array}{ll}
1 & \quad\textrm{ if }x\geq\varepsilon,\\
x/\varepsilon & \quad\textrm{ if }0\leq x\leq\varepsilon,\\
0 & \quad\textrm{ if }x\leq0.
\end{array}\right.
\]
Correspondingly, we consider the following regularized problem 
\begin{equation}
\left\{ \begin{array}{rcl}
\textrm{Find }p_{\varepsilon}\in L^{2}(\Omega)\textrm{ and }\mathbf{u_{\varepsilon}}\in H(div)\textrm{ such that }\\
p_{\varepsilon} & \geq & 0\qquad\textrm{ a.e. in }\Omega,\\
(\dfrac{12\mu}{\tau h^{3}}\mathbf{u_{\varepsilon}},\,\mathbf{v})-(\nabla\cdot\mathbf{v},\,p_{\varepsilon}) & = & -<p_{\Gamma},\,\mathbf{v}\cdot n>_{\Gamma}\qquad\forall\mathbf{v}\in H(div),\\
(\nabla\cdot\mathbf{u_{\varepsilon}},\,q)+(hH_{\varepsilon}(p_{\varepsilon}),\,q) & = & (h^{old}\theta^{old},\,q)\qquad\forall q\in L^{2}(\Omega).
\end{array}\right.\label{RegularizedProblem}
\end{equation}

\begin{lem}
The regularized problem $\eqref{RegularizedProblem}$ has a unique
solution $(\mathbf{u_{\varepsilon}},p_{\varepsilon}).$ Moreover there
exists a constant $C$ not depending on $\varepsilon$ such that 
\begin{equation}
\|\mathbf{u_{\varepsilon}}\|_{H(div)}+\|p_{\varepsilon}\|_{0}\leq C.\label{eq-Reg-Sol-Est}
\end{equation}
 
\end{lem}
\begin{proof}
Consider the mapping $T$ which, for any $p_{\varepsilon}\in L^{2}(\Omega),$
associates $\tilde{p}_{\varepsilon}=T(p_{\varepsilon})$ the solution
of the following problem: 
\begin{equation}
\begin{array}{rcl}
(\dfrac{12\mu}{\tau h^{3}}\mathbf{\widetilde{\mathbf{u}}_{\mathbf{\varepsilon}}},\,\mathbf{v})-(\nabla\cdot\mathbf{v},\,\tilde{p}_{\varepsilon}) & = & -<p_{\Gamma},\,\mathbf{v}\cdot n>_{\Gamma}\qquad\forall\mathbf{v}\in H(div),\\
(\nabla\cdot\mathbf{\widetilde{\mathbf{u}}_{\mathbf{\varepsilon}}},\,q)+(hH_{\varepsilon}(p_{\varepsilon}),\,q) & = & (h^{old}\theta^{old},\,q)\qquad\forall q\in L^{2}(\Omega),
\end{array}\label{LinearizedRegularizedProblem}
\end{equation}
The problem $\eqref{LinearizedRegularizedProblem}$ has a unique solution
by the inf-sup condition of Brezzi and Babuska \cite{Babuska,Brezzi}.
Moreover, there exists a constant $C=C(\Omega,\,p_{\Gamma},\,h_{0},\,h_{1})$
such that 
\begin{equation}
\|\widetilde{\mathbf{u}}_{\mathbf{\varepsilon}}\|_{H(div)}+\|\tilde{p}_{\varepsilon}\|_{0}\leq C(\Omega,\,p_{\Gamma},\,h_{0},\,h_{1},\,\mu,\,\tau).\label{Estimate1}
\end{equation}
Here, we have used the fact that $(H_{\varepsilon}(\widetilde{\mathbf{u}}_{\mathbf{\varepsilon}}))_{\varepsilon}>0$
is bounded in $L^{\infty}(\Omega)$ independently of $\varepsilon.$ 

Notice that the first equation in $\eqref{LinearizedRegularizedProblem}$
implies that
\[
\nabla\tilde{p}_{\varepsilon}=-\dfrac{12\mu}{\tau h^{3}}\mathbf{\widetilde{\mathbf{u}}_{\mathbf{\varepsilon}}}
\]
is a function in $L^{2}(\Omega)^{2}.$ This together with the estimate
$\eqref{Estimate1}$ shows that $T(p_{\varepsilon})=\tilde{p}_{\varepsilon}\in H^{1}(\Omega)$
and 
\begin{equation}
\|T(p_{\varepsilon})\|_{1}\leq C(\Omega,\,p_{\Gamma},\,h_{0},\,h_{1},\,\mu,\,\tau).\label{Estimate2}
\end{equation}
The mapping $T$ is then bounded from $L^{2}(\Omega)$ to $H^{1}(\Omega).$
From the compact embedding of $H^{1}(\Omega)$ into $L^{2}(\Omega)$,
it follows that $T$ is completely continuous from $H^{1}(\Omega)$
to $L^{2}(\Omega).$ Moreover the estimate $\eqref{Estimate2}$ shows
that $T(B_{C})\subset B_{C}$ with $B_{C}$ being the $H^{1}(\Omega)$-ball
of radius $C(\Omega,\,p_{\Gamma},\,h_{0},\,h_{1},\,\mu,\,\tau).$
Schauder fixed point theorem yields the existence of a function $p_{\varepsilon}$
such that $T(p_{\varepsilon})=p_{\varepsilon}.$ The fixed point $p_{\varepsilon}$
together with the corresponding $\mathbf{u_{\varepsilon}},$ forms
a solution of the two equations in $\eqref{RegularizedProblem}$ satisfying
(\ref{eq-Reg-Sol-Est}) with $C=C(\Omega,\,p_{\Gamma},\,h_{0},\,h_{1},\,\mu,\,\tau).$

Next, we claim that $p_{\varepsilon}\geq0$ a.e. in $\Omega.$ Actually,
$p_{\varepsilon}$ is also a solution of the following problem:

\begin{equation}
\left\{ \begin{array}{rcl}
\textrm{Find }p_{\varepsilon}\in H^{1}(\Omega)\textrm{ such that }\\
p_{\varepsilon} & = & p_{\Gamma}\qquad\textrm{ on }\Gamma,\\
hH_{\varepsilon}(p_{\varepsilon})-div(\dfrac{\tau h^{3}}{12\mu}\nabla p_{\varepsilon}) & = & h^{old}\theta^{old}\qquad\forall q\in H_{0}^{1}(\Omega).
\end{array}\right.\label{UsualRegularizedProblem}
\end{equation}
Let $p_{\varepsilon}^{-}=min(0,\,p_{\varepsilon}).$ It is clear that
$p_{\varepsilon}^{-}\in H_{0}^{1}(\Omega).$ By choosing $q=p_{\varepsilon}^{-}$
in $\eqref{UsualRegularizedProblem}$ we arrive at 
\[
(\dfrac{\tau h^{3}}{12\mu}\nabla p_{\varepsilon},\,\nabla p_{\varepsilon}^{-})+(hH_{\varepsilon}(p_{\varepsilon}),\,p_{\varepsilon}^{-})=(h^{old}\theta^{old},\,p_{\varepsilon}^{-}).
\]
Thus, using the fact that $H_{\varepsilon}(x)=0$ for $x\leq0$ and
$h^{old}\theta^{old}\geq0$ we obtain 
\[
(h^{3}\nabla p_{\varepsilon}^{-},\,\nabla p_{\varepsilon}^{-})\leq0,
\]
which leads to $p_{\varepsilon}^{-}=0$ a.e. in $\Omega$ and then
$p_{\varepsilon}\geq0$ a.e. in $\Omega.$ Consequently, the solution
$(\mathbf{u_{\varepsilon}},\,p_{\varepsilon})$ is a solution of $\eqref{RegularizedProblem}$.
\end{proof}
\begin{thm}
There exists a unique triple $(\mathbf{u},\,p,\,\theta)\in H(div)\times L^{2}(\Omega)\times L^{\infty}(\Omega)$
which is a solution of the weak formulation $\eqref{WeakMixedProblem}.$
\end{thm}
\begin{proof}
Proof. For any $\varepsilon>0,$ let $(\mathbf{u_{\varepsilon}},\,p_{\varepsilon})$
be the solution of $\eqref{RegularizedProblem}$. From (\ref{eq-Reg-Sol-Est})
we can find a subsequence, also denoted $(\mathbf{u_{\varepsilon}},\,p_{\varepsilon}),$
such that 
\[
\begin{array}{rl}
p_{\varepsilon}\rightharpoonup p & \textrm{ in }H^{1}(\Omega),\\
p_{\varepsilon}\longrightarrow p & \textrm{ in }L^{2}(\Omega),\\
\mathbf{u_{\varepsilon}}\rightharpoonup\mathbf{u} & \textrm{ in }H(div),\\
H_{\varepsilon}(p_{\varepsilon})\rightharpoonup\theta & \textrm{ in }L^{2}(\Omega),\\
H_{\varepsilon}(p_{\varepsilon})\overset{*}{\rightharpoonup}\theta & \textrm{ in }L^{\infty}(\Omega).
\end{array}
\]
By passing to the limit in $\eqref{RegularizedProblem}$, we deduce
that 
\begin{equation}
\left\{ \begin{array}{rcl}
p & \geq & 0\qquad\textrm{ a.e. in }\Omega,\\
(\dfrac{12\mu}{\tau h^{3}}\mathbf{u},\,\mathbf{v})-(\nabla\cdot\mathbf{v},\,p) & = & -<p_{\Gamma},\,\mathbf{v}\cdot n>_{\Gamma}\\
(\nabla\cdot\mathbf{u},\,q)+(h\theta,\,q) & = & (h^{old}\theta^{old},\,q).
\end{array}\right.
\end{equation}
To complete the proof of existence of a solution for the initial problem
$\eqref{WeakMixedProblem}$, it remains to prove that $0\leq\theta\leq1$
and $p(1-\theta)=0$ a.e. in $\Omega.$ \\
First, since 
\[
H_{\varepsilon}(p_{\varepsilon})\overset{*}{\rightharpoonup}\theta\textrm{ in }L^{\infty}(\Omega),
\]
we have 
\[
1-H_{\varepsilon}(p_{\varepsilon})\overset{*}{\rightharpoonup}1-\theta\textrm{ in }L^{\infty}(\Omega).
\]
Hence 
\[
\|\theta\|_{\infty}\leq\liminf H_{\varepsilon}(p_{\varepsilon})\leq1\textrm{\quad and}\quad\|1-\theta\|_{\infty}\leq\liminf(1-H_{\varepsilon}(p_{\varepsilon}))\leq1,
\]
i.e. 
\[
0\leq\theta\leq1.
\]
Furthermore, one has 
\begin{equation}
p_{\varepsilon}(1-H_{\varepsilon}(p_{\varepsilon}))\rightharpoonup p(1-\theta)\textrm{ in }L^{2}(\Omega).\label{p(1-theta)}
\end{equation}
Indeed, for $\phi$ in $L^{2}(\Omega)$ we have 
\[
\begin{array}{lcl}
(p_{\varepsilon}(1-H_{\varepsilon}(p_{\varepsilon}))-p(1-\theta),\:\phi) & = & ((p_{\varepsilon}-p)(1-H_{\varepsilon}(p_{\varepsilon})),\:\phi)+(p(\theta-H_{\varepsilon}(p_{\varepsilon})),\:\phi)\\
 & \leq & \|p_{\varepsilon}-p\|_{0}\|\phi\|_{0}+|((\theta-H_{\varepsilon}(p_{\varepsilon})),\:p\phi)|.
\end{array}
\]
Since $p\phi\in L^{1}(\Omega),$ we get $\eqref{p(1-theta)}$ from
the $L^{2}(\Omega)$ strong convergence of $p_{\varepsilon}$ to $p$
and the $L^{\infty}(\Omega)$ weak-{*} convergence of $H_{\varepsilon}(p_{\varepsilon})$
to $\theta.$

On the other hand, from $H_{\varepsilon}$ expression, we have 
\[
(p_{\varepsilon}(1-H_{\varepsilon}(p_{\varepsilon})),\:\phi)\leq\varepsilon\|\phi\|_{0},
\]
whence 
\begin{equation}
p_{\varepsilon}(1-H_{\varepsilon}(p_{\varepsilon}))\rightharpoonup0\textrm{ in }L^{2}(\Omega).
\end{equation}
Consequently, from the uniqueness of the limit in $L^{2}(\Omega),$
we deduce that $p(1-\theta)=0$ a.e. in $\Omega.$ \\
Finally, notice that the complementarity system 
\begin{equation}
\left\{ \begin{array}{l}
\nabla\cdot\mathbf{u}+h\theta=h^{old}\theta^{old}\qquad\text{ in }L^{2}(\Omega),\\
p\geq0,\quad p(1-\theta)=0,\quad1-\theta\geq0,\qquad\textrm{ a.e. in }\Omega,
\end{array}\right.\label{eq-complementarity}
\end{equation}

and the variational inequality 
\begin{equation}
\left\{ \begin{array}{l}
p\in L_{+}^{2}(\Omega):=\left\{ q\in L^{2}(\Omega):\quad q\geq0\right\} ,\\
(\nabla\cdot\mathbf{u},\,p-q)\leq(h^{old}\theta^{old}-h,\,p-q),\qquad\forall q\in L_{+}^{2}(\Omega)
\end{array}\right.\label{eq-VI}
\end{equation}
are equivalent. Therefore, for any $(\mathbf{u},\,p,\,\theta)$ satisfying
(\ref{WeakMixedProblem}), $(\mathbf{u},\,p)$ is a solution of the
the following mixed variational inequality
\begin{equation}
\left\{ \begin{array}{l}
p\in L_{+}^{2}(\Omega):=\left\{ q\in L^{2}(\Omega):\quad q\geq0\right\} ,\\
(\dfrac{12\mu}{\tau h^{3}}\mathbf{u},\,\mathbf{v})-(\nabla\cdot\mathbf{v},\,p)=-<p_{\Gamma},\,\mathbf{v}\cdot n>_{\Gamma},\qquad\forall\mathbf{v}\in H(div),\\
(\nabla\cdot\mathbf{u},\,p-q)\leq(h^{old}\theta^{old}-h,\,p-q),\qquad\forall q\in L_{+}^{2}(\Omega),
\end{array}\right.\label{Mixed-VI}
\end{equation}

which has a at most one solution by virtue of \cite[Theorem 2.1]{Brezzi-al}.
The uniqueness of $\theta$ follows from (\ref{eq-complementarity})
and (\ref{thikness}).
\end{proof}
The weak formulation $\eqref{WeakMixedProblem}$ allows us to approximate
the underlying lubrication problem by using mixed finite element methods.
To this end, let $\mathcal{T}_{h}$ be a finite element partition
of $\Omega$ into triangles. We consider a stable pair $(\mathcal{V}_{h},\,\mathcal{Q}_{h})$
of spaces on $\mathcal{T}_{h}$ such that $p_{h}\in\mathcal{Q}_{h}\subset L^{2}(\Omega)$
and $\mathbf{u}_{h}\in\mathcal{V}_{h}\subset H(div).$ 

Our mixed finite element problem consists in finding a triple $(p_{h},\,\mathbf{u}_{h},\,\theta_{h})$
in $\mathcal{Q}_{h}\times\mathcal{V}_{h}\times\mathcal{Q}_{h}$ satisfying
\begin{equation}
\left\{ \begin{array}{l}
(\dfrac{12\mu}{\tau h^{3}}\mathbf{u}_{h},\,\mathbf{v}_{h})-(\nabla\cdot\mathbf{v}_{h},\,p_{h})=-<p_{\Gamma},\,\mathbf{v}_{h}\cdot n>_{\Gamma}\qquad\forall\mathbf{v}_{h}\in\mathcal{V}_{h},\\
(\nabla\cdot\mathbf{u}_{h},\,q_{h})+(h\theta_{h},\,q_{h})=(h^{old}\theta_{h}^{old},\,q_{h})\qquad\forall q_{h}\in\mathcal{Q}_{h},\\
p_{h}\geq0,\quad p_{h}(1-\theta_{h})=0,\quad0\leq\theta_{h}\leq1,\qquad\textrm{ a.e. in }\Omega.
\end{array}\right.\label{DiscreteWeakMixedProblem}
\end{equation}

\begin{thm}
The discrete problem $\eqref{DiscreteWeakMixedProblem}$ has a unique
solution.
\end{thm}
\begin{proof}
By considering the following discrete regularized problem 
\begin{equation}
\left\{ \begin{array}{rcl}
\textrm{Find }p_{h,\varepsilon}\in\mathcal{Q}_{h}\textrm{ and }\mathbf{u_{h,\varepsilon}}\in\mathcal{V}_{h}\textrm{ such that }\\
p_{h,\varepsilon} & \geq & 0\qquad\textrm{ a.e. in }\Omega,\\
(\dfrac{12\mu}{\tau h^{3}}\mathbf{u_{h,\varepsilon}},\,\mathbf{v}_{h})+(\nabla\cdot\mathbf{v}_{h},\,p_{h,\varepsilon}) & = & -<p_{\Gamma},\,\mathbf{v}_{h}\cdot n>_{\Gamma}\qquad\forall\mathbf{v}_{h}\in\mathcal{V}_{h},\\
(\nabla\cdot\mathbf{u_{h,\varepsilon}},\,q_{h})+(hH_{\varepsilon}(p_{h,\varepsilon}),\,\mathbf{v}_{h}) & = & (h^{old}\theta^{old},\,q_{h})\qquad\forall q_{h}\in\mathcal{Q}_{h},
\end{array}\right.\label{DiscreteRegularizedProblem}
\end{equation}
this theorem can be proven similarly to the previous one. In particular
it can be shown that the solution of $\eqref{DiscreteWeakMixedProblem}$
is given by the unique solution of the following discrete mixed variational
inequality
\begin{equation}
\left\{ \begin{array}{l}
p_{h}\in\mathcal{Q}_{h}^{+}:=\left\{ q_{h}\in\mathcal{Q}_{h}:\quad q_{h}\geq0\right\} ,\\
(\dfrac{12\mu}{\tau h^{3}}\mathbf{u}_{h},\,\mathbf{v}_{h})-(\nabla\cdot\mathbf{v}_{h},\,p_{h})=-<p_{\Gamma},\,\mathbf{v}_{h}\cdot n>_{\Gamma}\qquad\forall\mathbf{v}_{h}\in\mathcal{V}_{h},\\
(\nabla\cdot\mathbf{u}_{h},\,p_{h}-q_{h})+(h\theta_{h},\,p_{h}-q_{h})\leq(h^{old}\theta^{old},\,p_{h}-q_{h})\qquad\forall q_{h}\in\mathcal{Q}_{h}^{+}.
\end{array}\right.\label{DiscreteVI}
\end{equation}
\end{proof}

\section{Active set strategy}

The solution algorithm we are proposing in this section relies on
the complementarity formulation of (\ref{DiscreteVI}). More precisely
we consider the following system
\[
\left\{ \begin{array}{l}
(\dfrac{12\mu}{\tau h^{3}}\bm{u}_{h},\,\mathbf{v}_{h})+(\nabla\cdot\mathbf{v}_{h},\,p_{h})=-<p_{\Gamma},\,\mathbf{v}_{h}\cdot n>_{\Gamma}\qquad\forall\mathbf{v}_{h}\in\mathcal{V}_{h},\\
(\nabla\cdot\bm{u}_{h},\,q_{h})+(\lambda_{h},\,q_{h})=(\lambda^{old}+h-h^{old},\,q_{h})\qquad\forall q_{h}\in\mathcal{Q}_{h},\\
p_{h}\geq0,\quad p_{h}\lambda_{h}=0,\quad\lambda_{h}\geq0,\qquad\textrm{ a.e. in }\Omega.
\end{array}\right.
\]
For sake of convenience we have made the following change of variables
\[
\lambda_{h}=h(1-\theta_{h})\qquad\bm{u}_{h}=-\mathbf{u}_{h}\qquad\lambda^{old}=h^{old}(1-\theta^{old}).
\]

The matrix form of (\ref{DiscreteVI}) reads

\begin{equation}
\left\{ \begin{array}{l}
MU+BP=F_{v}\\
B^{T}U+D\Lambda=F_{q}\\
P\geq0,\quad P^{T}\Lambda=0,\quad\Lambda\geq0,
\end{array}\right.\label{MatrixDiscreteMixedSystem}
\end{equation}
where $U,$ $P$ and $\Lambda$ are the algebraic representation of
$\bm{u_{h}},$ $p_{h}$ and $\lambda_{h}$ with respect to the basis
of $\mathcal{V}_{h}$ and $\mathcal{Q}_{h}.$ The matrix blocks and
the right side vectors are given by
\[
M=\left[(\dfrac{12\mu}{\tau h^{3}}\mathbf{v_{i}},\,\mathbf{v_{j}})\right]_{i,j=1}^{n_{u}}\qquad B=\left[(\nabla\cdot\mathbf{v_{j}},\,q_{i})\right]_{i,j=1}^{n_{p},n_{u}}\qquad D=\left[(\,q_{i},\,q_{j})\right]_{i,j=1}^{n_{p}}
\]
\[
F_{v}=\left[<p_{\Gamma},\,\mathbf{v_{i}}\cdot\bm{n}>_{\Gamma}\right]_{i=1}^{n_{u}},\qquad F_{q}=\left[(\lambda^{old}+(h-h^{old}),\,q_{i})\right]_{i=1}^{n_{p}}.
\]
By $\left(\mathbf{v_{i}}\right)_{i=1}^{n_{u}}$ and $\left(q_{i}\right)_{i=1}^{n_{p}}$
we denote the basis functions of the spaces $\mathcal{V}_{h}$ and
$\mathcal{Q}_{h}$ respectively. 

Now for a fixed parameter $c>0,$ we introduce the active and inactive
sets 
\[
A:=\{i\in\{1,\dots,n_{p}\}\,|\,\Lambda_{i}-cP_{i}>0\},\qquad I:=\{i\in\{1,\dots,n_{p}\}\,|\,\Lambda_{i}-cP_{i}\leq0\}.
\]
Accordingly we split the quantities 
\[
P=\left(\begin{array}{c}
P_{I}\\
P_{A}
\end{array}\right),\quad\Lambda=\left(\begin{array}{c}
\Lambda_{I}\\
\Lambda_{A}
\end{array}\right),\quad F_{q}=\left(\begin{array}{c}
F_{qI}\\
F_{qA}
\end{array}\right),
\]
\[
B=\left(\begin{array}{cc}
B_{\bullet I} & B_{\bullet A}\end{array}\right),\quad D=\left(\begin{array}{cc}
D_{I} & 0\\
0 & D_{A}
\end{array}\right).
\]
Here and in the following $D$ is assumed to be diagonal. In the numerical
experiments Raviart-Thomas elements of lowest order $RT_{0}$ (2d-example)
or Taylor-Hood $P_{2}-P_{1}$ elements with a lumped $P_{1}$-mass
matrix are used (1d- and 2d-examples). 

Next we state the primal-dual active set strategy for solving (\ref{MatrixDiscreteMixedSystem}). 

\subsection*{Algorithm}
\begin{enumerate}
\item Choose $c>0.$ Initialize $P^{0},$ $\Lambda^{0}.$ Set $k=0.$
\item Set $A^{k}=\{i\in\{1,\dots,n_{p}\}\,|\,\Lambda_{i}^{k}-cP_{i}^{k}>0\},$
 $I^{k}=\{i\in\{1,\dots,n_{p}\}\,|\,\Lambda_{i}^{k}-cP_{i}^{k}\leq0\}.$ 
\item Solve for $\left(U^{k+1},\,P_{I}^{k+1}\right):$ 
\[
\left(\begin{array}{cc}
M & B_{\bullet I}\\
B_{\bullet I}^{T} & 0
\end{array}\right)\left(\begin{array}{c}
U^{k+1}\\
P_{I}^{k+1}
\end{array}\right)=\left(\begin{array}{c}
F_{v}\\
F_{qI}
\end{array}\right)
\]
and set 
\[
P_{A}^{k+1}=0,\quad\Lambda_{I}^{k+1}=0,\quad\Lambda_{A}^{k+1}=D_{A}^{-1}\left(F_{qA}-B_{\bullet A}^{T}U^{k+1}\right)
\]
\item Stop, or set $k=k+1,$ and return to 2.
\end{enumerate}
Hereafter we briefly comment this algorithm. For more details about
the primal-dual active set strategy with their local convergence properties
we refer to \cite{HintermItoKunisch} and the references therein.
\begin{itemize}
\item For the solution $\left(U,\,P,\,\Lambda\right)$ we have $P_{A}=0$
and $\Lambda_{I}=0.$ However since the active and inactive sets depend
on the solution itself, $A$ and $I$ are estimated during the iterations
based on the current iterate.
\item The above algorithm corresponds to a semi-smooth Newton method applied
to the nonlinear system (\ref{MatrixDiscreteMixedSystem}). In particular,
the complementarity relation 
\begin{equation}
P\geq0,\quad P\cdot\Lambda=0,\quad\Lambda\geq0,\label{complementarity}
\end{equation}
is interpreted as
\[
\bm{\Theta}(P,\,\Lambda)=\left(\Theta(P{}_{1},\,\Lambda_{1}),\cdots,\Theta(P_{n_{p}},\,\Lambda_{n_{p}})\right)=0,
\]
where $\Theta(a,b)=a-\max(0,a-cb)$ for $a,$ $b\in\mathbb{R}$ and
a fixed $c>0.$ \\
It is worth mentioning that a similar ideas has been used in \cite{Lengiewicz,Woloszynski}
with a standard finite element discretization.
\item Step 3. corresponds to the update step where a ''linearization''
of the $\max$-term is computed by using the generalized derivative
\[
\partial\max(0,\,P-c\Lambda)=\operatorname{diag}(d_{1},\ldots,d_{n_{p}})\quad\text{with }d_{i}=\left\{ \begin{array}{ll}
0 & \text{ if }P_{i}-c\Lambda_{i}\leq0,\\
1 & \text{ if }P_{i}-c\Lambda_{i}>0.
\end{array}\right.
\]
The involved linear system has a standard saddle point form. In our
numerical tests, the built-in backslash Matlab solver gives excellent
performances. However, for very fine meshes direct solvers becomes
unfeasible and iterative methods have to be investigated. For more
details about saddle point problem algorithms we refer to \cite{Benzi,Elman}.
\item The value of parameter $c$ does not influence the convergence of
the algorithm once we are sufficiently close to the solution. However
$c$ should be tuned carefully if a globalization strategy is used.
In the context of our time-dependant problem and for a reasonable
time step size, the solution at a given time is expected to provide
a good initial point for computing the solution at the next time step. 
\item As a stopping criterion we use $A^{k}=A{}^{k+1}.$ 
\end{itemize}

\section{Numerical experiments}

In this section, we assess the practical performance of the proposed
time-space discretization scheme and the solution algorithm described
above. For this purpose, a Matlab code was written using the Getfem++
library \cite{getfem}. Examples from the literature \cite{Ausas,Lengiewicz}
have been considered for validation and - for comparison purposes
- publicly available source codes is used \cite{Almqvist-new,Ausas}.

\subsection{Sinusoidal bearing profile}

In our first example, we consider a steady-state one-dimensional problem
with a hydrodynamic bearing of sinusoidal shape. The film thickness
is defined as 
\[
h(x)=h_{av}-\Delta h\cos(2\pi x/l),\qquad x\in(-l/2,\,l/2)
\]
where $h_{av}=0.02$ mm, $\Delta h=0.005$ mm and $l=125$ mm. The
sliding speed is $U=4$ m/s, the lubricant viscosity is $\mu=0.015$
Pa.s and the cavitation pressure is $p_{c}=0$ MPa. At the boundary
the pressure is prescribed by $p(\pm l/2)=1$ MPa. The computation
is stopped as soon as the difference in the maximum norm between two
successive pressures is less than $10^{-12}$ with a time step $\tau=10^{-4}.$
Figure \ref{fig-ex1-sol} depicts the numerical solutions computed
on uniform grids of size $dx=1/1000.$ An excellent agreement has
been found between the present formulation and the one proposed in
\cite{Giacopini-et-al}. 

\begin{figure}
\begin{centering}
\includegraphics[width=10cm,height=6cm]{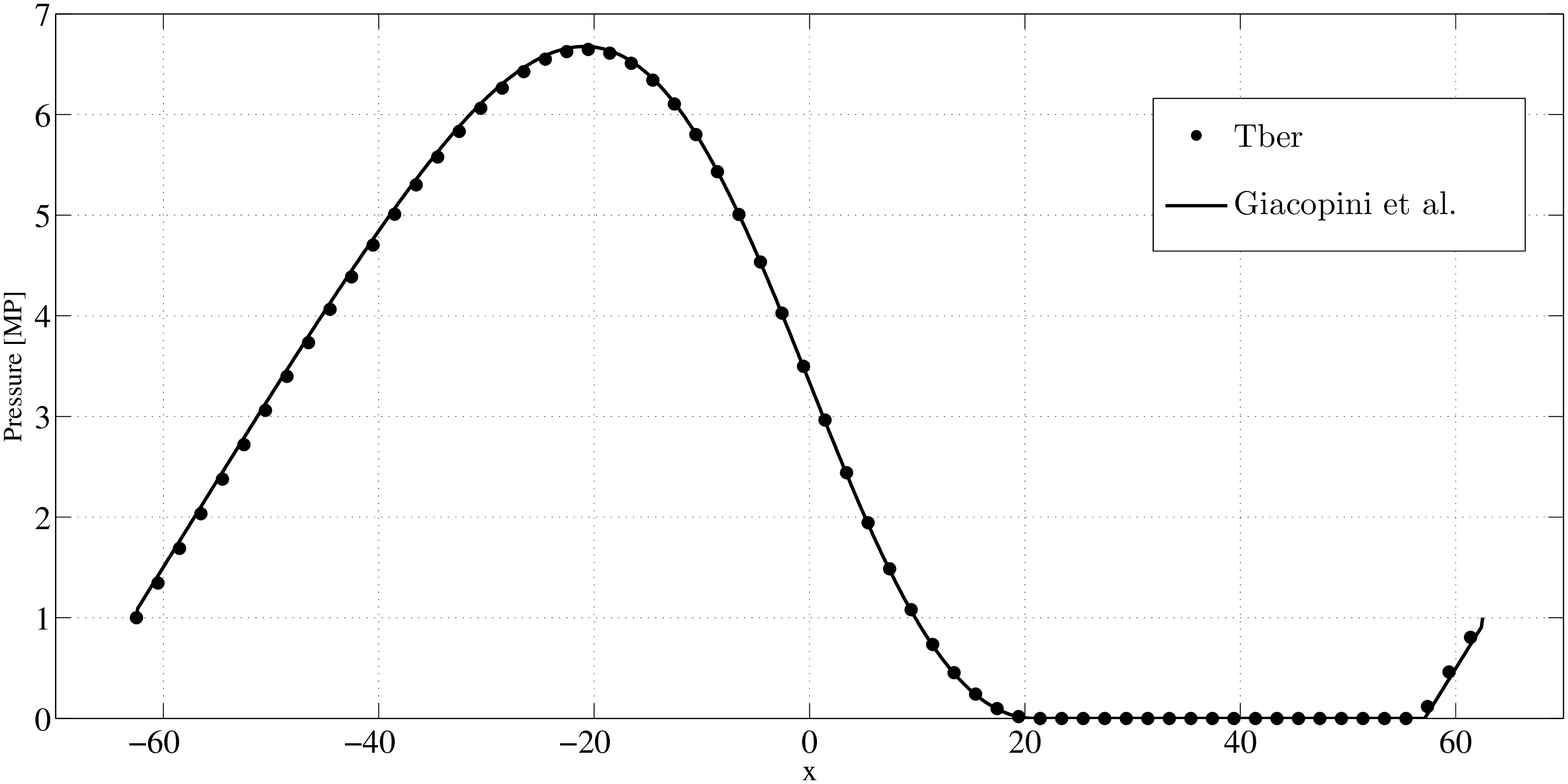}
\par\end{centering}
\begin{centering}
\includegraphics[width=10cm,height=4cm]{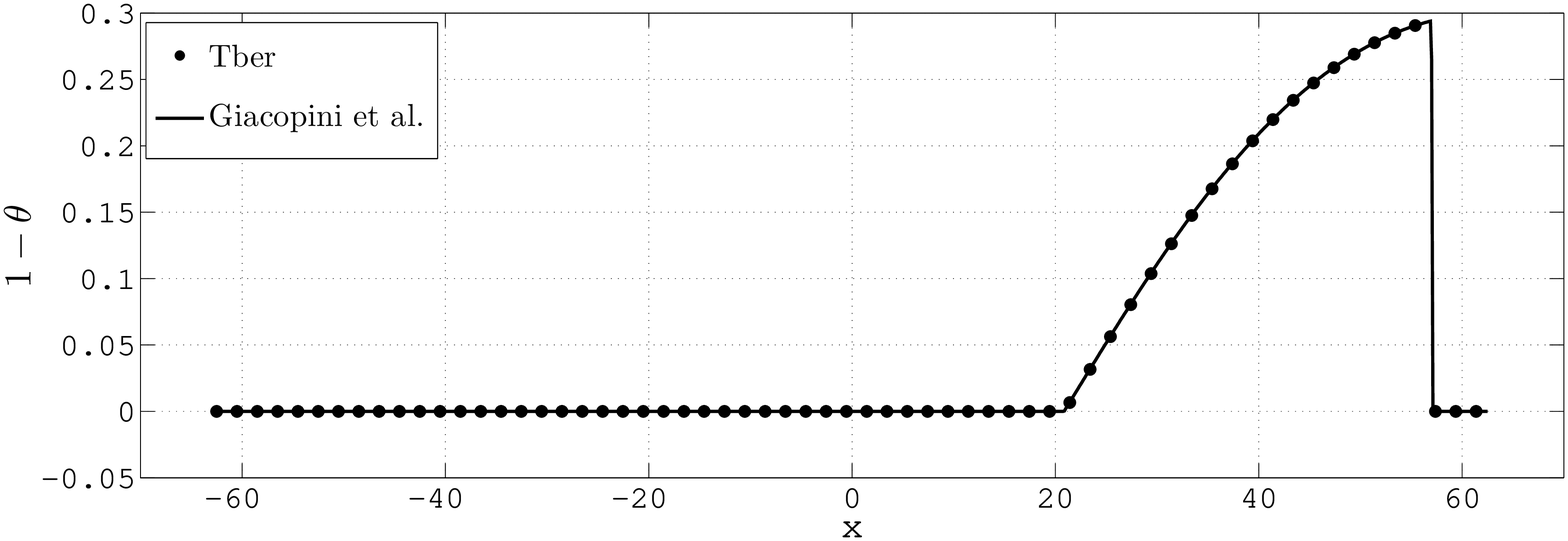}
\par\end{centering}
\caption{\label{fig-ex1-sol}Solution of the 1D sinusoidal bearing problem:
pressure (top) and void fraction (bottom)}

\end{figure}

\subsection{Oscillatory squeeze flow }

For the transient case we consider the classical example of two parallel
plates in pure squeeze motion $(U=0)$ and separated by a lubricant
with constant viscosity. The dimensionless formulation as presented
in \cite{Ausas} is used with a computational domain $\Omega=(0,\,1)$
and a film thickness $h(x,t)=H(t)=0.125\cos(4\pi t)+0.375.$ On the
boundary $p(0,\,t)=p(1,\,t)=0.025.$ The initial condition is given
by $\theta_{0}(x)=1.$ For the numerical solution, a $450$-elements
domain has been employed and the simulation time has been divided
into $3000$ steps. Figure \ref{fig-ex2-sol} shows the variation
in time of the extent of the cavitation zone. In comparison to Ausas
et al. algorithm and the analytic solutions given in \cite{Ausas}
a very good agreement was achieved. As mentioned there, the step character
of the numerical result is due to discretization and could be made
smoother by increasing the numbers of time steps and the points used
for the discretization of the radius of cavitation area. In most stages
of the simulation the active set solver converges in few iterations,
though a large number of iterations is necessary during the transition
from no-cavitation to with-cavitation phases.

\begin{figure}
\begin{centering}
\includegraphics[width=8cm,height=4cm]{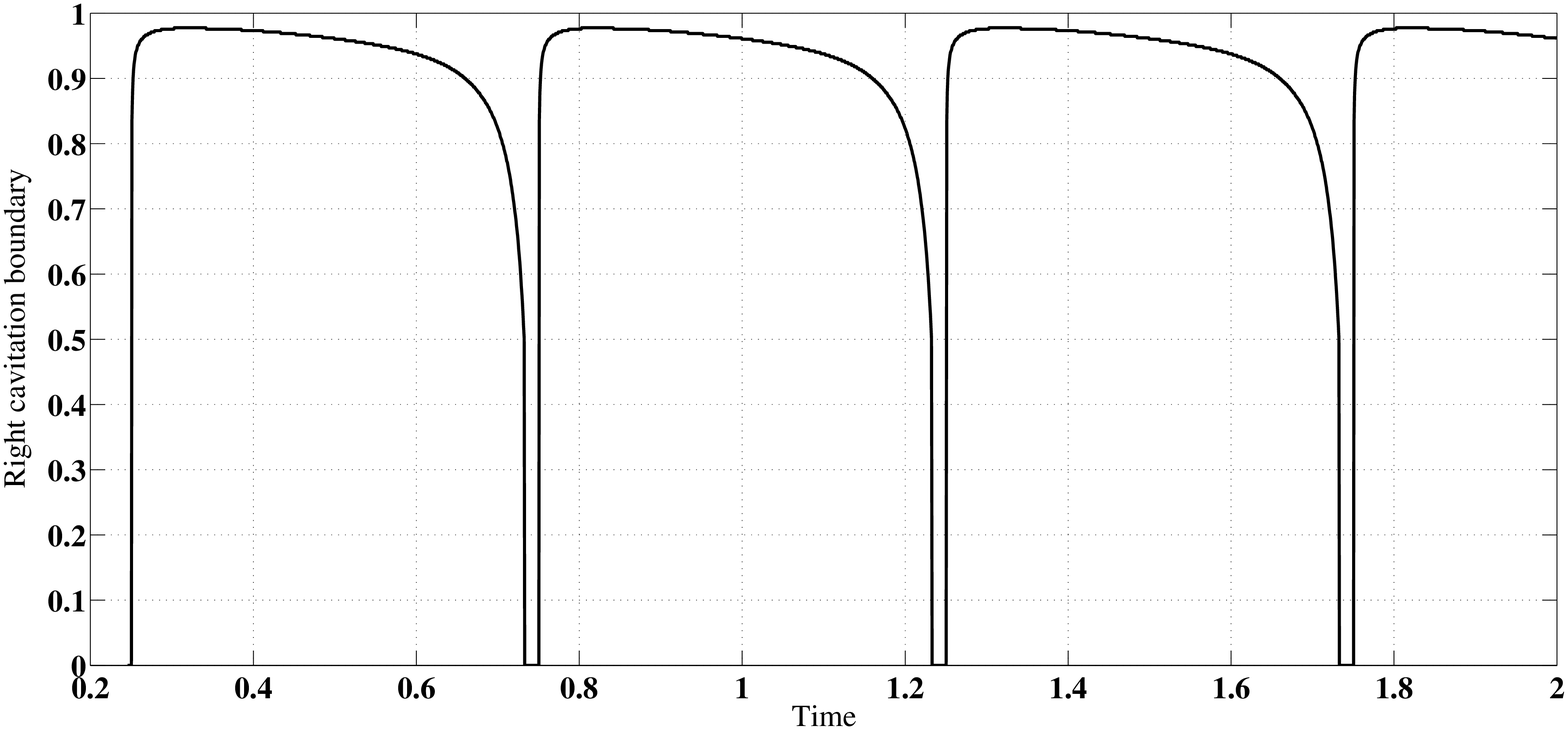}
\par\end{centering}
\begin{centering}
\includegraphics[width=8cm,height=4cm]{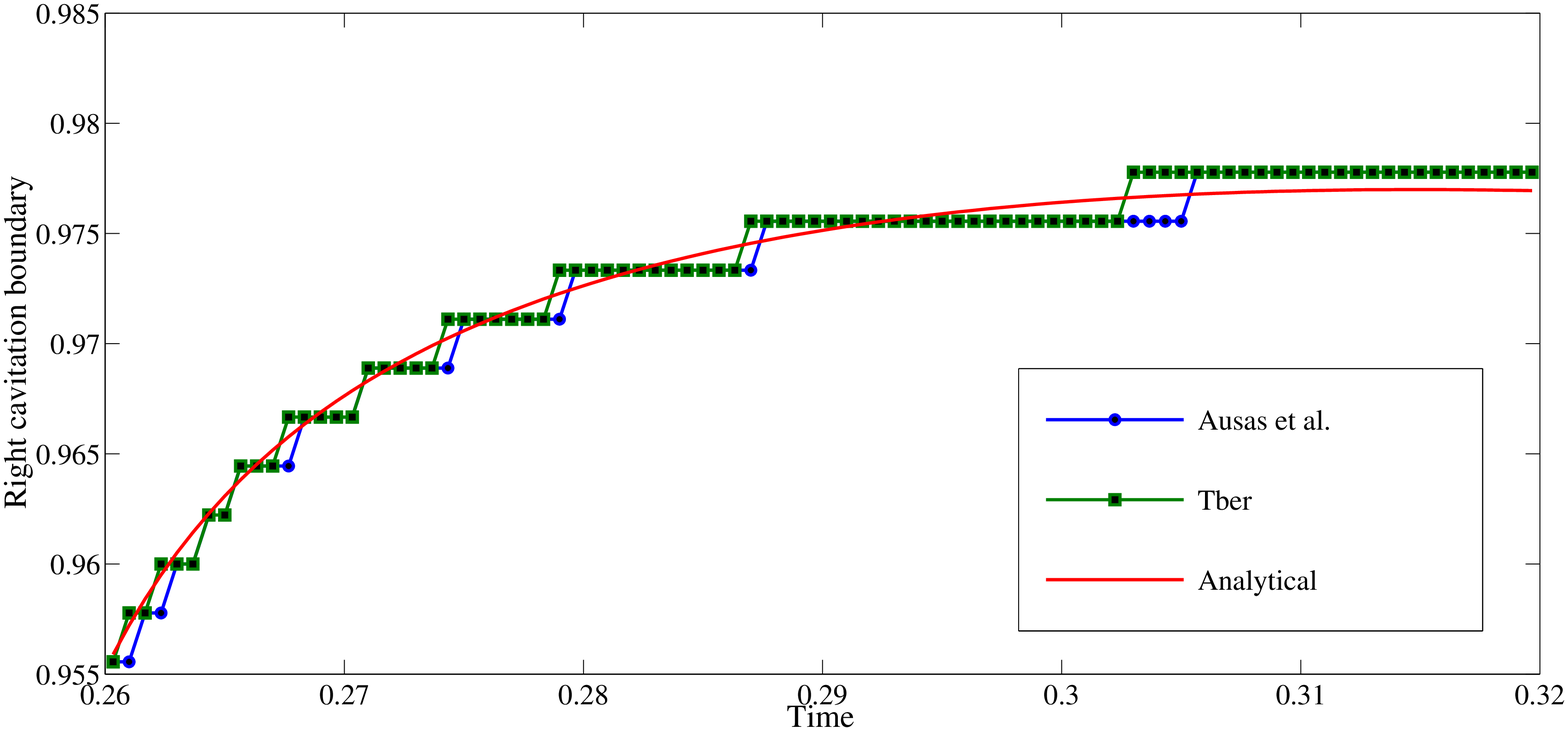}
\par\end{centering}
\begin{centering}
\includegraphics[width=8cm,height=4cm]{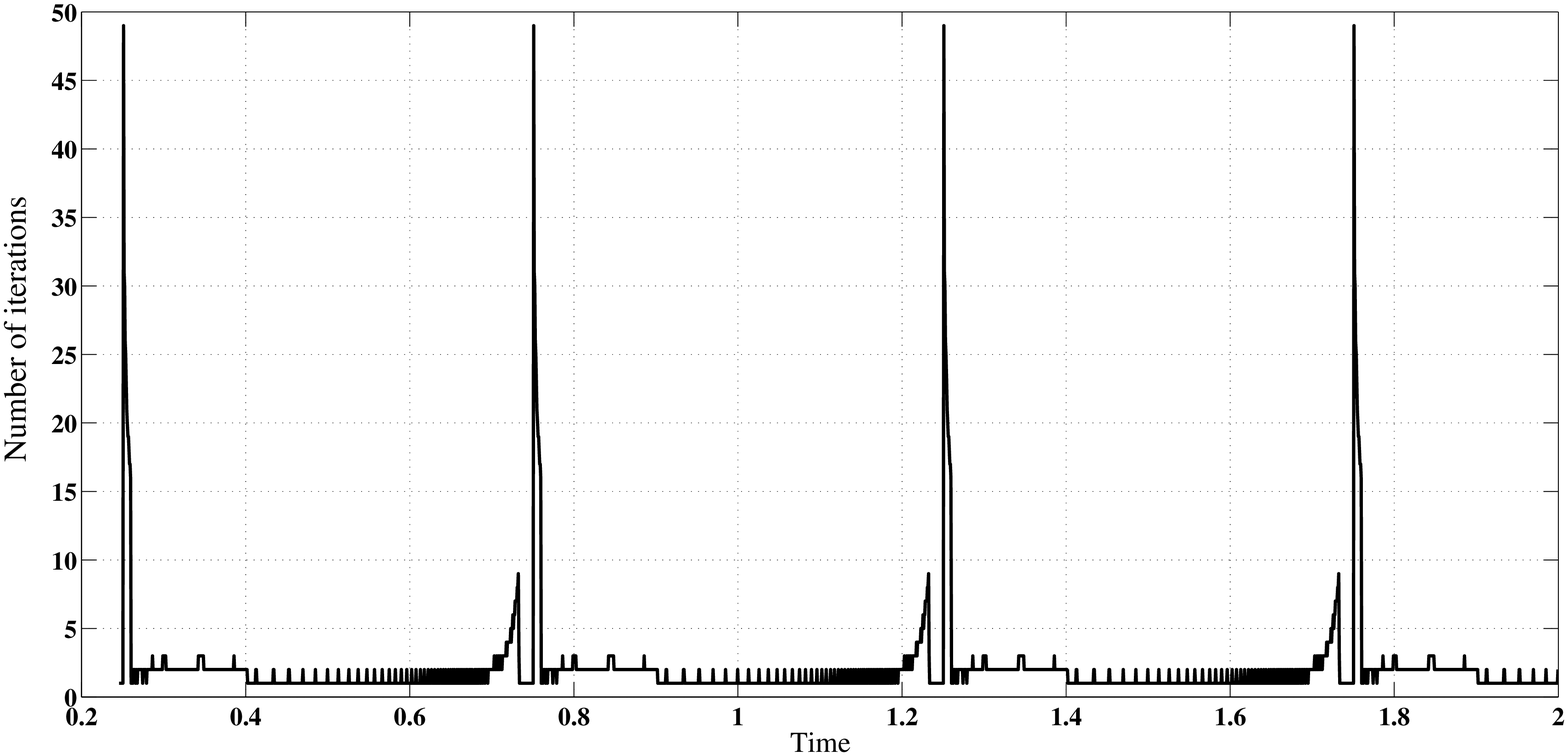}
\par\end{centering}
\caption{\label{fig-ex2-sol}Pure squeeze problem: (top) time evolution of
the right cavitation boundary (middle) comparison with Ausas et al.
algorithm and the analytic method (bottom) number of the active set
algorithm iterations with respect to time steps.}
\end{figure}

\subsection{Sinusoidal bearing profile in 2D}

The sinusoidal bearing profile of the first example is extended here
to two dimensions by considering 
\[
h(x)=h_{av}-\Delta h\cos(2\pi x_{1}/l),\qquad x\in(-l/2,\,l/2)\times(-l/2,\,l/2)
\]
and a constant pressure on the whole boundary $p=1$ MPa. The steady
state pressure and void fraction profiles obtained using $RT_{0}$
elements on a structured mesh with $100\times100$ elements are shown
in figure \ref{fig-ex3-sol}. In the same figure, the number of iterations
of the active set algorithm until successful termination is provided
at each time step for different grids. Up to the first time step,
the proposed algorithm converges in less than $10$ iterations regardless
of the mesh size. For the first time step a globalization strategy
might be investigated to deal with the large number iterations. In
fact, the standard result on the convergence of semi\textendash smooth
Newton methods requires an initial iteration sufficiently close to
the solution. In figure \ref{fig-ex3-sol-TH} the computational results
obtained using Taylor-Hood elements are shown. For this particular
example better results are obtained. With $RT_{0}$ elements, the
void fraction profile exhibits oscillations in the active-inactive
sets boundary. Naturally with $P2-P1$ elements the cost in terms
of computation time and memory is higher. 

\begin{figure}
\begin{centering}
\includegraphics[width=10cm,height=6cm]{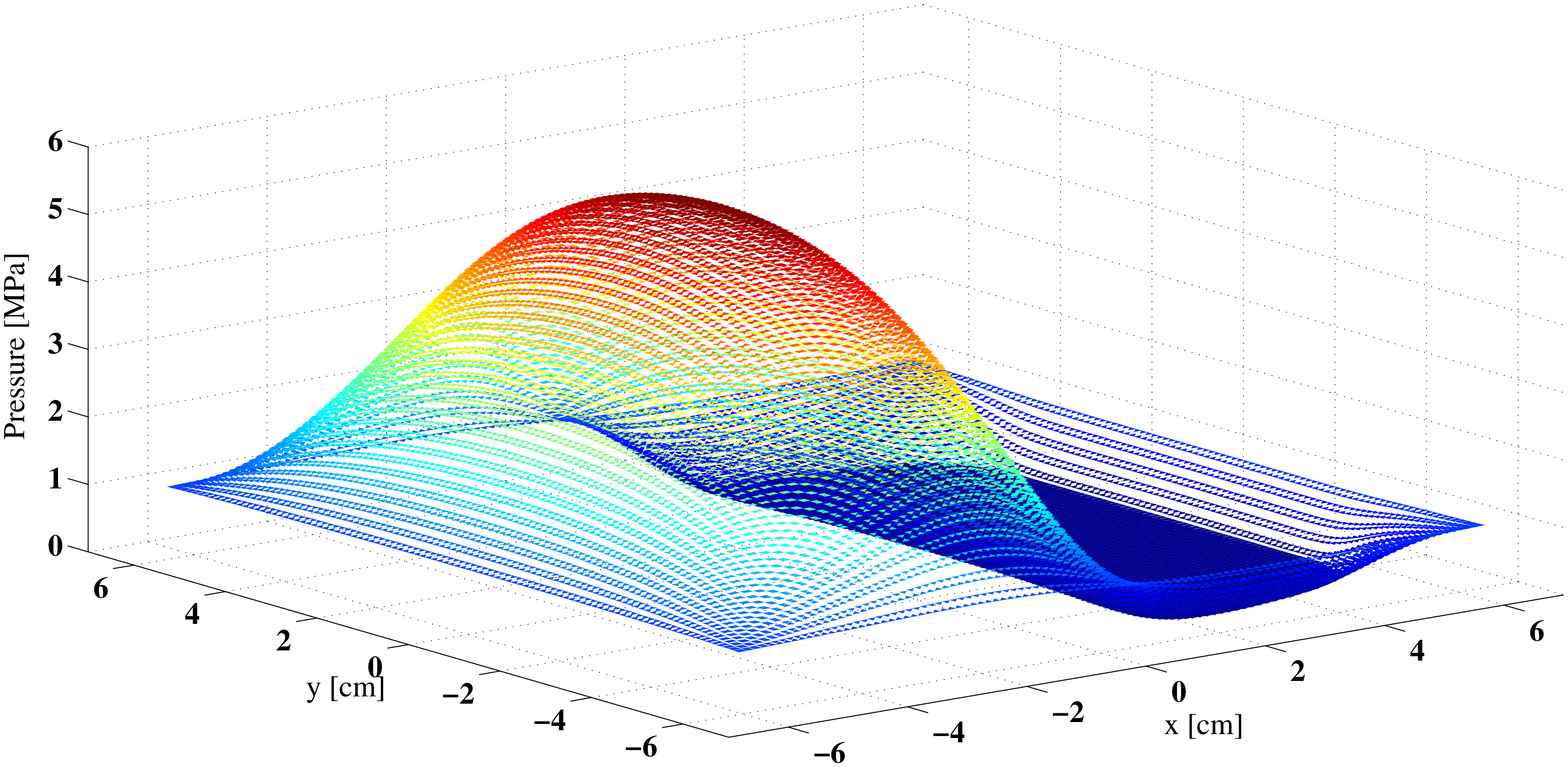}
\par\end{centering}
\begin{centering}
\includegraphics[width=10cm,height=6cm]{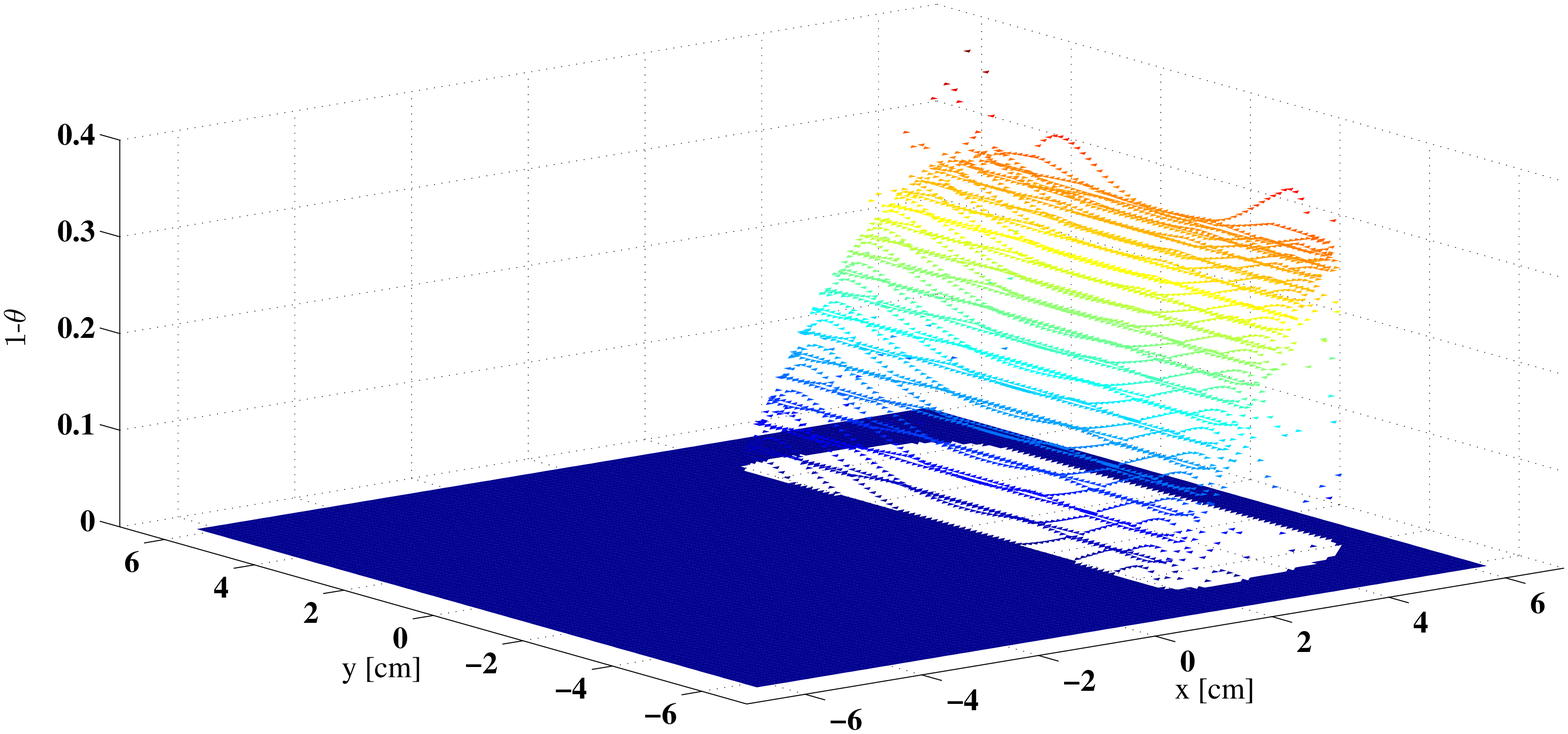}
\par\end{centering}
\begin{centering}
\includegraphics[width=10cm,height=6cm]{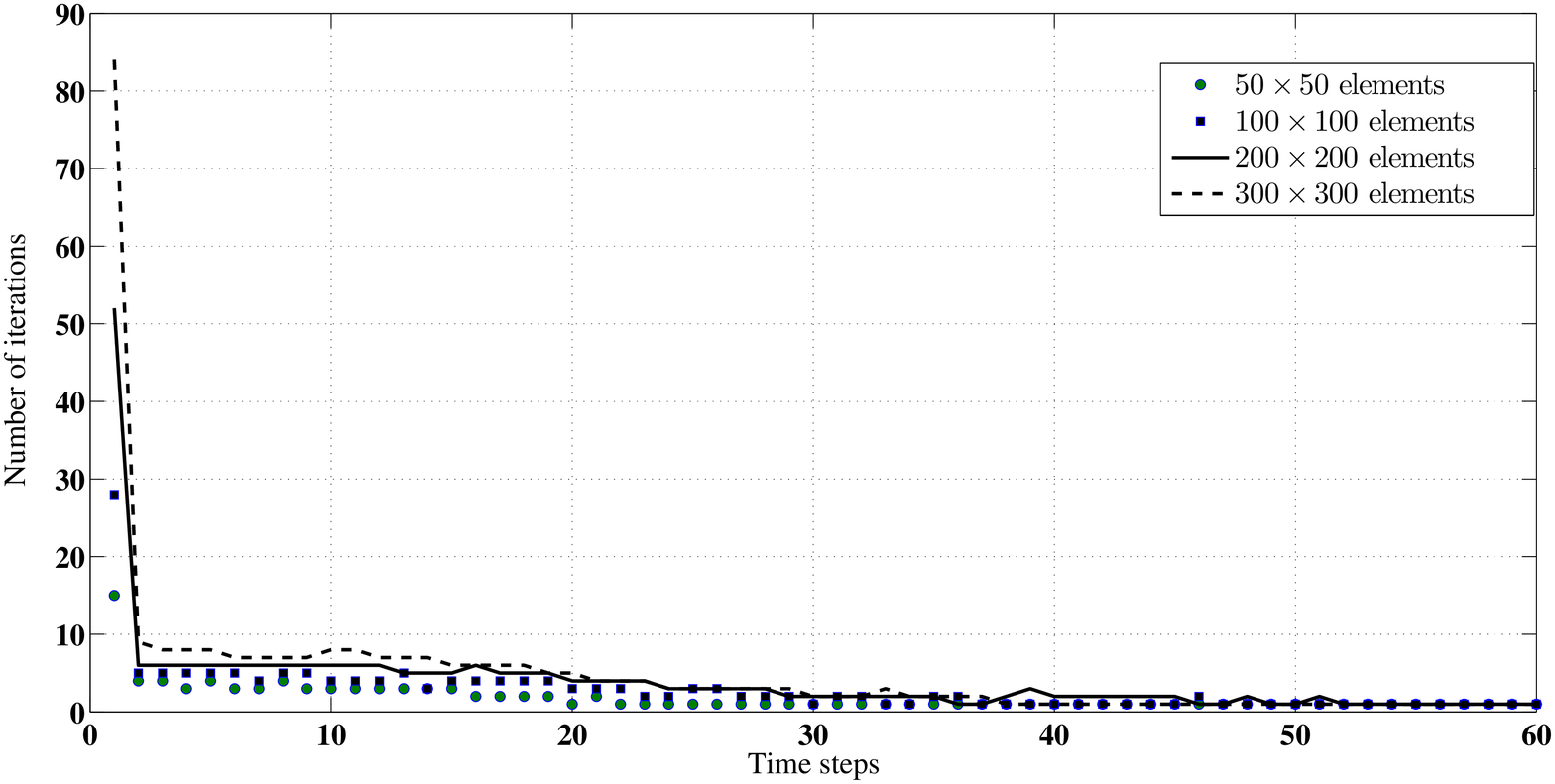}
\par\end{centering}
\caption{\label{fig-ex3-sol}2D sinusoidal profile: (top) pressure (middle)
void fraction (bottom) number of the active set solver iterations
with respect to time steps.}
\end{figure}

\begin{figure}
\begin{centering}
\includegraphics[width=10cm,height=6cm]{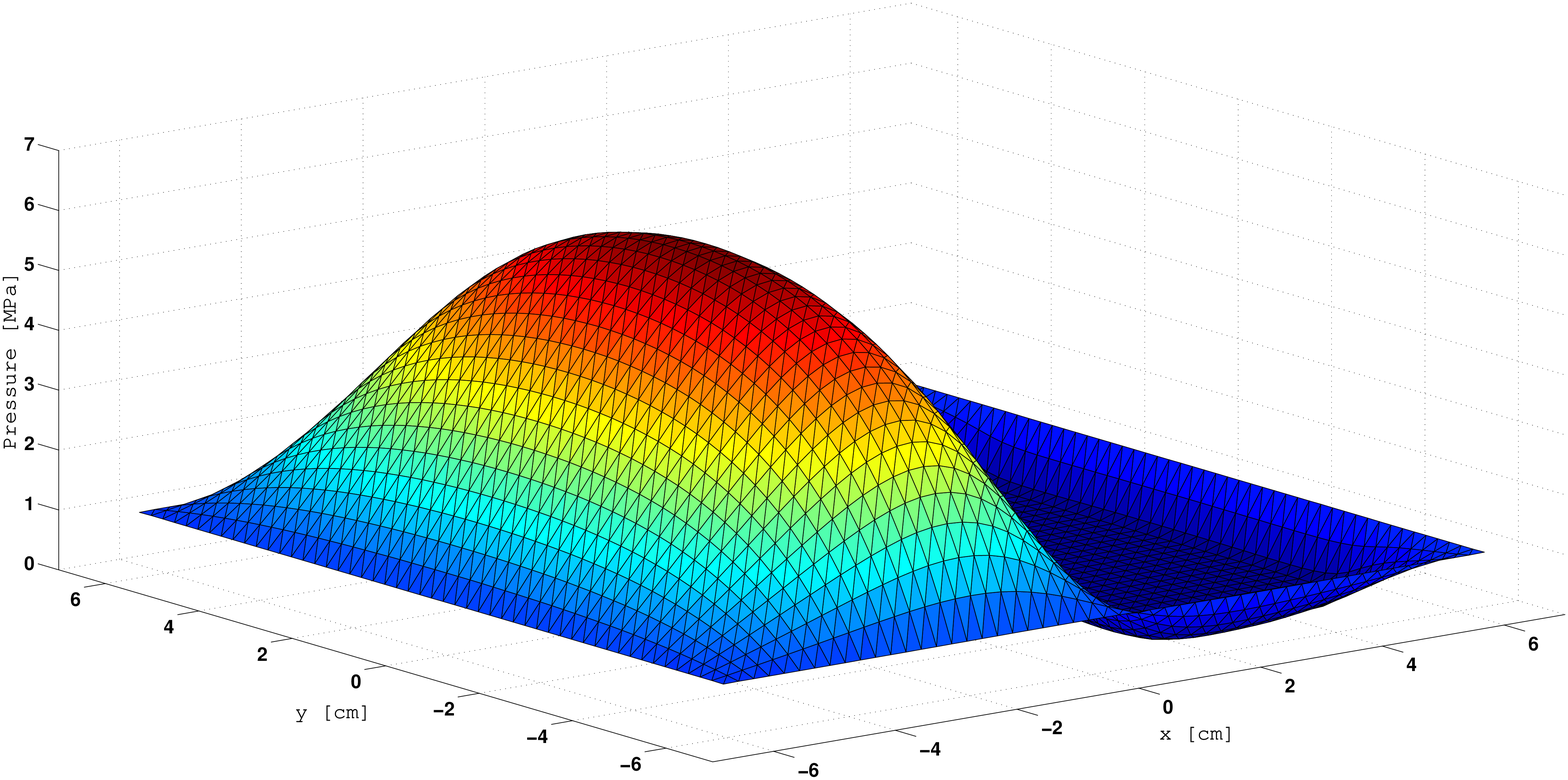}
\par\end{centering}
\begin{centering}
\includegraphics[width=10cm,height=6cm]{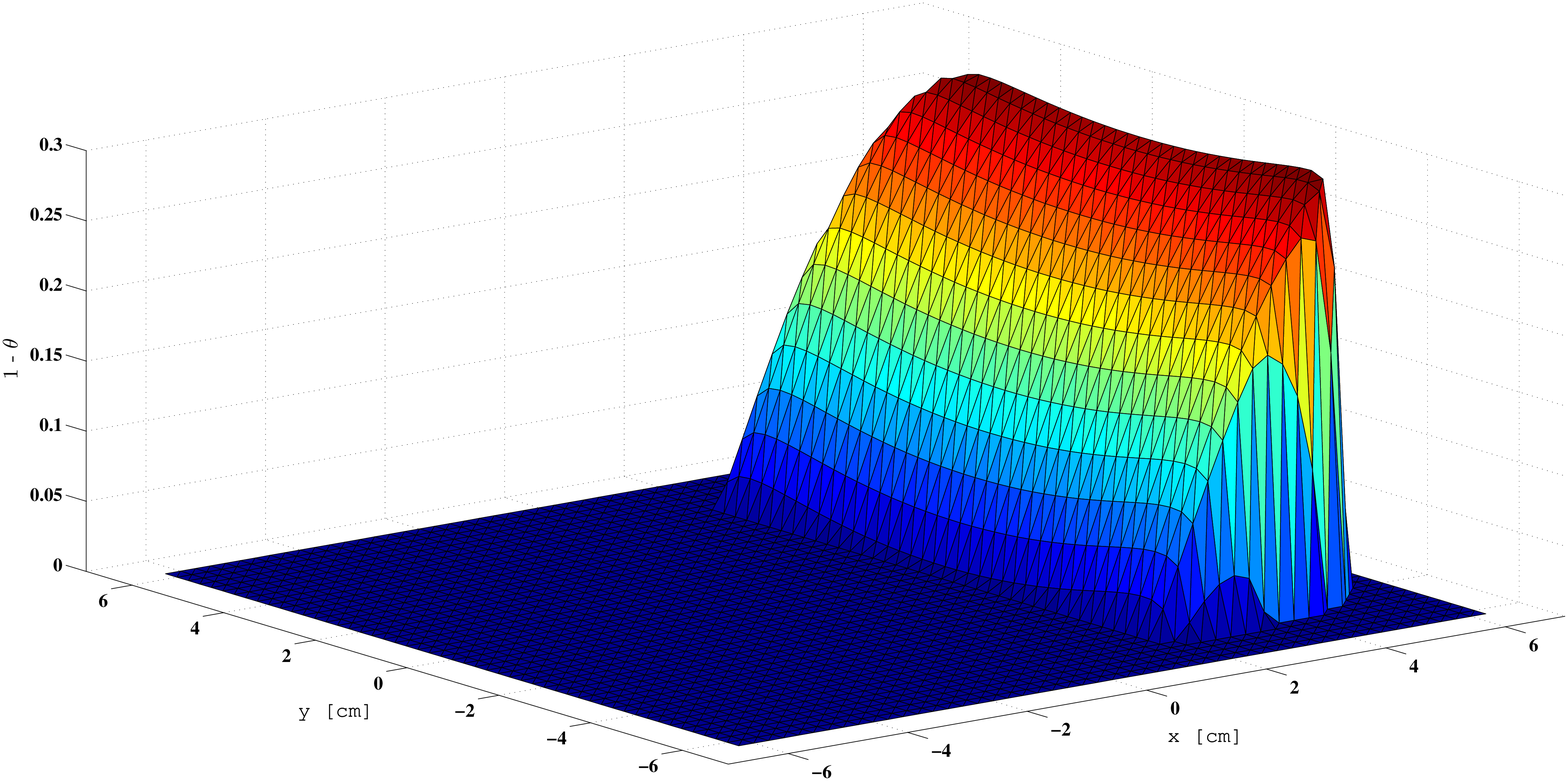}
\par\end{centering}
\centering{}\caption{\label{fig-ex3-sol-TH}2D sinusoidal profile with Taylor-Hoods elements:
pressure (top) void fraction (bottom).}
\end{figure}

\end{document}